\documentclass[psamsfonts,fceqn,leqno]{amsart}
\usepackage{mathrsfs,latexsym,amsfonts,amssymb,curves,epic}
\usepackage{color}

\newtheorem{theorem}{Theorem}[section]
\newtheorem{corollary}[theorem]{Corollary}
\newtheorem{proposition}[theorem]{Proposition}
\newtheorem{lemma}[theorem]{Lemma}
\newtheorem{question}[theorem]{Question}

\theoremstyle{definition}
\newtheorem{definition}[theorem]{Definition}

\def\N{{\mathbb N}}

\begin{document}
\title[The $k$-property and countable tightness of free topological vector spaces]
{The $k$-property and countable tightness of free topological vector spaces}

  \author{Fucai Lin}
  \address{(Fucai Lin): School of mathematics and statistics,
  Minnan Normal University, Zhangzhou 363000, P. R. China}
  \email{linfucai2008@aliyun.com; linfucai@mnnu.edu.cn}

\author{Shou Lin}
\address{(Shou Lin): Institute of Mathematics, Ningde Teachers' College, Ningde, Fujian
352100, P. R. China} \email{shoulin60@163.com}

  \author{Chuan Liu}
\address{(Chuan Liu): Department of Mathematics,
Ohio University Zanesville Campus, Zanesville, OH 43701, USA}
\email{liuc1@ohio.edu}

  \thanks{The first author is supported by the NSFC (Nos. 11571158, 11471153),
  the Natural Science Foundation of Fujian Province (Nos. 2017J01405, 2016J05014, 2016J01671, 2016J01672) of China and the Program for New Century Excellent Talents in Fujian Province University.}

  \keywords{free topological vector space; $k$-space; $k_{\omega}$-space; countable tightness; separable.}
  \subjclass[2000]{46A03; 22A05; 54A25; 54D50.}

  \begin{abstract}
The free topological vector space $V(X)$ over a Tychonoff space $X$ is a pair consisting of a topological vector space $V(X)$ and a continuous map $i=i_{X}: X\rightarrow V(X)$ such that every continuous mapping $f$ from $X$ to a topological vector space $E$ gives rise to a unique continuous linear operator $\overline{f}: V(X)\rightarrow E$ with $f=\overline{f}\circ i$. In this paper the $k$-property and countable tightness of free topological vector space over some generalized metric spaces are studied. The characterization of a space $X$ is given such that the free topological vector space $V(X)$ is a $k$-space or the tightness of $V(X)$ is countable. Furthermore, the characterization of a space $X$ is also provided such that if the fourth level of $V(X)$ has the $k$-property or is of the countable tightness then $V(X)$ is too.
  \end{abstract}

  \maketitle

\section{Introduction}
The free topological group $F(X)$, the free abelian topological group $A(X)$ and the free locally convex space $L(X)$ over a Tychonoff space $X$ were introduced by Markov \cite{MA1945} and intensively studied over the last half-century, see for example \cite{AOP1989,Gab2014, Gab20141,GM2017, LLL,LL, LLLL,U1991,Y1993,Y1997,Y1998}.
Recently, in \cite{GM2017} S.S. Gabriyelyan and S.A. Morris introduced and studied the free topological vector space $V(X)$ over a Tychonoff space $X$. One surprising fact is that the free topological vector spaces in some respect behave better than the free locally convex spaces. For example, if $X$ is a $k_{\omega}$-space then $V(X)$ is also a $k_{\omega}$-space \cite{GM2017}; however, for a Tychonoff space $X$, the space $L(X)$ is a $k$-space if and only if $X$ is a countable discrete space \cite{Gab2014}. Therefore, it is natural to consider the following question:

\begin{question}\label{q1}
If $V(X)$ is a $k$-space over a Tychonoff space $X$, is $V(X)$ a $k_{\omega}$-space?
\end{question}

Recently, S.S. Gabriyelyan proved that for a metrizable space $X$ the free locally convex $L(X)$ is of countable tightness if and only if $X$ is separable, see \cite{Gab2014}. Hence it is is natural to consider the following question:

\begin{question}\label{q2}
Let $X$ be a metrizable space. Is the countable tightness of $V(X)$ equivalent to the separability of $X$?
\end{question}

In this paper, we shall give an affirmative answer to Question~\ref{q1} if $X$ is a $k^*$-metrizable space, and an affirmative answer to Question~\ref{q2} if $X$ is a paracompact $k$-and $\sigma$-space. Moreover, the characterization of a space $X$ is also given such that if the fourth level of $V(X)$ has the $k$-property or is of the countable tightness then $V(X)$ is too.

The paper is organized as follows. In Section 2, we introduce the necessary notation and terminologies which are
  used for the rest of the paper. In Section 3, we investigate the
  $k$-property and countable tightness on free topological vector spaces. In section 4, we pose some interesting questions about the free topological vector spaces which are still unknown to us.

  \section{Notation and Terminologies}
In this section, we introduce the necessary notation and terminologies. Throughout this paper, all topological spaces are assumed to be
  Tychonoff and all vector spaces are over the field of real number $\mathbb{R}$, unless otherwise is explicitly stated.
  First of all, let $\N$ be the set of all positive
  integers and $\omega$ the first infinite ordinal. Let $X$ be a topological space and $A \subseteq X$ be a subset of $X$.
  The \emph{closure} of $A$ in $X$ is denoted by $\overline{A}$; moreover, we always denote the set of all the non-isolated points of $X$ by $X^{\prime}$. For undefined
  notation and terminologies, the reader may refer to \cite{AT2008},
  \cite{E1989}, \cite{Gr1984} and \cite{Lin2015}.

  \medskip
  \begin{definition}
  Let $X$ be a space.
 \begin{enumerate}
\item The space $X$ is {\it separable} if it contains a countable subset $B$ such that $\overline{B}=X$.
\item The space $X$ is of {\it countable tightness} if the closure of any subset $A$ of $X$ equals the union of closures of all countable subsets of $A$, and the countable tightness of a space $X$ is denoted by $t(X)\leq\omega$.
\item The space $X$ is called a
  \emph{$k$-space} provided that a subset $C\subseteq X$ is closed in $X$ if
  $C\cap K$ is closed in $K$ for each compact subset $K$ of $X$.
\item  The space $X$ is called a \emph{$k_{\omega}$-space} if there exists
  a family of countably many compact subsets $\{K_{n}: n\in\mathbb{N}\}$ of
  $X$ such that each subset $F$ of $X$ is closed in $X$ provided that
  $F\cap K_{n}$ is closed in $K_{n}$ for each $n\in\mathbb{N}$.
\item A subset $P$ of $X$ is called a
  \emph{sequential neighborhood} of $x \in X$, if each
  sequence converging to $x$ is eventually in $P$. A subset $U$ of
  $X$ is called \emph{sequentially open} if $U$ is a sequential neighborhood of
  each of its points. A subset $F$ of
  $X$ is called \emph{sequentially closed} if $X\setminus F$ is sequentially open. The space $X$ is called a \emph{sequential space} if each
  sequentially open subset of $X$ is open.
\end{enumerate}
  \end{definition}

Let $\kappa$ be an infinite cardinal.
  For each $\alpha\in\kappa$, let $T_{\alpha}$ be a sequence converging to
  $x_{\alpha}\not\in T_{\alpha}$. Let $T=\bigoplus_{\alpha\in\kappa}(T_{\alpha}\cup\{x_{\alpha}\})$ be the topological sum of $\{T_{\alpha}
  \cup \{x_{\alpha}\}: \alpha\in\kappa\}$. Then
  $S_{\kappa}=\{x\}  \cup \bigcup_{\alpha\in\kappa}T_{\alpha}$
  is the quotient space obtained from $T$ by
  identifying all the points $x_{\alpha}\in T$ to the point $x$.

  A space $X$ is called an \emph{ $S_{2}$}-{space} ({\it Arens' space})  if
$$X=\{\infty\}\cup \{x_{n}: n\in \mathbb{N}\}\cup\{x_{n, m}: m, n\in
\omega\}$$ and the topology is defined as follows: Each
$x_{n, m}$ is isolated; a basic neighborhood of $x_{n}$ is
$\{x_{n}\}\cup\{x_{n, m}: m>k\}$, where $k\in\omega$;
a basic neighborhood of $\infty$ is $$\{\infty\}\cup (\bigcup\{V_{n}:
n>k\ \mbox{for some}\ k\in \omega\}),$$ where $V_{n}$ is a
neighborhood of $x_{n}$ for each $n\in\omega$.

 \medskip
  \begin{definition}
  Let $X$ be a space and $\mathscr P$ a family of subsets of $X$.
 \begin{enumerate}
\item The family $\mathscr{P}$ is a \emph{network} of $X$ if for each $x\in X$ and $x\in U$ with $U$ open in
$X$, then $x\in P\subset U$ for some $P\in\mathscr P_x$. A regular space $X$ is called a \emph{$\sigma$-space} (resp. \emph{cosmic space}) if it has a $\sigma$-locally finite network (resp. countable network).
\item The family $\mathscr P$ is called
  a \emph{$cs$-network} \cite{G1971} at a point $x\in X$ if for every sequence
  $\{x_n: n \in \N\}$ converging to $x$ and an arbitrary open neighborhood
  $U$ of $x$ in $X$ there exist an $m\in\N$ and an element $P\in
  \mathscr P$ such that
  \[
  \{x\}\cup \{x_n: n\geqslant m\} \subseteq P \subseteq U.
  \]
The space $X$ is called \emph{csf-countable} if $X$ has a countable $cs$-network
  at each point $x\in X$.
\item The family $\mathscr P$ is called a {\it $k$-network}
  \cite{O1971} if for every compact subset $K$ of $X$ and an arbitrary open set
  $U$ containing $K$ in $X$ there is a finite subfamily $\mathscr {P}^{\prime}
  \subseteq \mathscr {P}$ such that $K\subseteq \bigcup\mathscr {P}^{\prime}
  \subseteq U$. A regular space $X$ is called an {\it $\aleph$-space} (resp. {\it $\aleph_{0}$-space}) if it has a $\sigma$-locally finite $k$-network (resp. countable $k$-network).
\item Let $\mathscr P$ be a cover of $X$ such that (i)
  $\mathscr P = \bigcup_{x\in X}\mathscr{P}_{x}$; (ii) for each point $x\in X$,
  if $U,V\in \mathscr{P}_{x}$, then $W\subseteq U\cap V$ for some $W\in
  \mathscr{P}_{x}$; and (iii) for each point $x\in X$ and each open neighborhood
  $U$ of $x$ there is some $P\in\mathscr P_x$ such that $x\in P \subseteq U$.
  Then, $\mathscr P$ is called an \emph{sn-network} \cite{Lin1996} for $X$ if
  for each point $x\in X$, each element of $\mathscr P_x$ is a sequential
  neighborhood of $x$ in $X$, and $X$ is called \emph{snf-countable}
  \cite{Lin1996} if $X$ has an $sn$-network $\mathscr P$ and $\mathscr P_x$
  is countable for all $x\in X$.
\end{enumerate}
  \end{definition}

Clearly, the following implications follow directly
  from definitions:
  \[
  \mbox{first countable} \Rightarrow \mbox{$snf$-countable} \Rightarrow
  \mbox{$csf$-countable}.
  \]
  Note that none of the above implications can be reversed. It is well known
  that $S_{\omega}$ is $csf$-countable but not
  $snf$-countable and $S_{\omega_{1}}$ is not $csf$-countable.

A Hausdorff topological space $X$ is $k^*$-metrizable if $X$ is the image of a
metrizable space $M$ under a continuous map $f: M\rightarrow X$ having a section $s: X\rightarrow M$ that
preserves precompact sets in the sense that the image $s(K)$ of any compact set $K\subset X$ has
compact closure in $X$. In \cite{BBK2008}, T. Banakh, V.I. Bogachev and A.V. Kolesnikov introduced this concept of $k^*$-metrizable spaces, and systematically studied this class of $k^*$-metrizable spaces. The class of $k^*$-metrizable spaces is closed under many countable (and some uncountable)
topological operations. Regular $k^*$-metrizable spaces can be characterized as spaces with
$\sigma$-compact-finite $k$-network, see \cite[Theorem 6.4]{BBK2008}. This characterization shows that the class
of $k^*$-metrizable spaces is sufficiently wide and contains all La\v{s}nev spaces (closed images
of metrizable spaces), all $\aleph_{0}$-spaces and all $\aleph$-spaces.
$k^*$-metrizable spaces form a new class
of generalized metric spaces and have various applications in topological
algebra, functional analysis, and measure theory, see \cite{BBK2008, G2013}. From \cite{BBK2008}, we list some properties of $k^*$-metrizable spaces.

(1) The $k^*$-metrizability is preserved by subspaces, countable product, box-product and topological sum;

(2) Each sequentially compact subset of a $k^*$-metrizable space is metrizable;

(3) A $k^*$-metrizable space is sequential if and only if it is a $k$-space.

(4) A regular space is $k^*$-metrizable space if and only if it has a $\sigma$-compact-finite $k$-network;

(5) A regular $X$ is a $k^*$-metrizable space with countable network if and only if it is an $\aleph_{0}$-space.

\smallskip

\begin{definition}\cite{GM2017}
The free topological vector space $V(X)$ over a Tychonoff space $X$ is a pair consisting of a topological vector space $V(X)$ and a continuous map $i=i_{X}: X\rightarrow V(X)$ such that every continuous mapping $f$ from $X$ to a topological vector space (tvs) $E$ gives rise to a unique continuous linear operator $\overline{f}: V(X)\rightarrow E$ with $f=\overline{f}\circ i$.
\end{definition}

The change of the word ``topological vector space'' to ``abelian topological group'' and ``locally convex space'' in the above definition gives the definition of the free abelian topological group $A(X)$ and free locally convex space $L(X)$ on $X$ respectively.

For a space $X$ and an arbitrary $n\in\mathbb{N}$, we denote by $\mbox{sp}_{n}(X)$ the following subset of $V(X)$
$$\mbox{sp}_{n}(X)=\{\lambda_{1}x_{1}+\cdots+\lambda_{n}x_{n}: \lambda_{i}\in[-n, n], x_{i}\in X, i=1, \ldots, n\}.$$Then $V(X)=\bigcup_{n\in\mathbb{N}}\mbox{sp}_{n}(X)$ and each $\mbox{sp}_{n}(X)$ is closed in $V(X)$, see \cite[Theorem 2.3]{GM2017}. Moreover, If $v\in V(X)$, then $v$ has a unique representation $$v=\lambda_{1}x_{1}+\cdots+\lambda_{n}x_{n},\ \mbox{where}\ \lambda_{i}\in\mathbb{R}\setminus\{0\}\ \mbox{and}\ x_{i}\in X\ \mbox{are distinct};$$then the set $\mbox{supp}(v)=\{x_{1}, \ldots, x_{n}\}$ is called the {\it support} \cite{GM2017} of the element $v$. For every $n\in\mathbb{N}$, define the mapping $T_{n}: [-n, n]^{n}\times X^{n}\rightarrow\mbox{sp}_{n}(X)$ by $$T_{n}((a_{1}, \cdots, a_{n})\times (x_{1}, \cdots, x_{n}))=a_{1}x_{1}+\cdots+a_{n}x_{n}$$ for arbitrary $(a_{1}, \cdots, a_{n})\times (x_{1}, \cdots, x_{n})\in [-n, n]^{n}\times X^{n}$.

\section{main results}
First of all, we give a characterization of a collectionwise normal $\aleph$-space $X$ such that $V(X)$ is an $\aleph_{0}$-space, and obtain that $\mbox{sp}_{1}(X)$ is $csf$-countable if and only if $V(X)$ is an $\aleph_{0}$-space if and only if $X$ is separable. In order to prove this result, we need two lemmas and some concepts.

\begin{lemma}\label{p0}
If $X$ is Dieudom\'{e}-complete and $\phi$ is a compact set in $V(X)$, then there exist a compact set $Z\subset X$ and $n\in\mathbb{N}$ such that $\phi$ is the continuous image of some compact subspace in $[-n, n]^{n}\times Z^{n}$.
\end{lemma}

\begin{proof}
By \cite[Proposition 5.5]{GM2017}, the set $Z=\overline{\mbox{supp}(\phi)}$ is compact in $X$ and there exists $n\in\mathbb{N}$ such that $\phi\subset\mbox{sp}_{n}(X)$. Consider the mapping $$\sigma=T_{n}|_{([-n, n]^{n}\times Z^{n})\cap T_{n}^{-1}(\phi)}: ([-n, n]^{n}\times Z^{n})\cap T_{n}^{-1}(\phi)\rightarrow \phi.$$ Then $\sigma$ is a continuous onto mapping from a compact subspace $([-n, n]^{n}\times Z^{n})\cap T_{n}^{-1}(\phi)$ to $\phi$, hence $\phi$ is the continuous image of some compact subspace in $[-n, n]^{n}\times Z^{n}$.
\end{proof}

\begin{lemma}\label{t00}
If $X$ is an $\aleph_{0}$-space, then $V(X)$ is also an $\aleph_{0}$-space.
\end{lemma}

\begin{proof}
For each $n\in\mathbb{N}$, let $Y_{n}=[-n, n]^{n}\times X^{n}$ and fix a countable $k$-network $\mathscr{P}_{n}$ in $Y_{n}$ since $Y_{n}$ is an $\aleph_{0}$-space. Since $X$ is Dieudom\'{e}-complete, it follows from Lemma~\ref{p0} that for each compact set $\phi\subset \mbox{sp}_{n}(X)$ there exists a compact set $\phi_{1}\subset Y_{n}$ such that $T_{n}(\phi_{1})=\phi$. For each $n\in\mathbb{N}$, since $T_{n}$ is continuous, it easily see that the family $\mathscr{F}_{n}=\{T_{n}(P): P\in\mathscr{P}_{n}\}$ is a countable $k$-network of the subspace $\mbox{sp}_{n}(X)$. Therefore, it follows from \cite[Corollary 3.4]{GM2017} that the family $\mathscr{F}=\bigcup\{\mathscr{F}_{n}: n\in\mathbb{N}\}$ is a countable $k$-network of $V(X)$.
\end{proof}

Let $\kappa$ be an infinite cardinal, let $\mathbb{V}_{\kappa}=\bigoplus_{i<\kappa}\mathbb{R}_{i}$ be the direct sum of $\kappa$ copies of $\mathbb{R}$, and let $\tau_{\kappa}$, $\nu_{\kappa}$ and $\mu_{\kappa}$ be the box topology, maximal locally convex vector topology and maximal vector topology on $\mathbb{V}_{\kappa}$, respectively. Obviously, $\tau_{\kappa}\subset\nu_{\kappa}\subset\mu_{\kappa}$ and $V(D)\cong (\mathbb{V}_{\kappa}, \mu_{\kappa})$, where $D$ is a discrete space of cardinality $\kappa$. Indeed, let $D=\{x_{\alpha}: \alpha<\kappa\}$. Then define the mapping $f: V(D)\rightarrow (\mathbb{V}_{\kappa}, \mu_{\kappa})$ by $$f(\lambda_{1}x_{\alpha_{1}}+\ldots+\lambda_{n}x_{\alpha_{n}})=(y_{\beta})_{\beta<\kappa},$$ where $y_{\alpha_{i}}=\lambda_{i}, 1\leq i\leq n$ and $y_{\beta}=0$ if $\beta\not\in\{\alpha_{i}:1\leq i\leq n\}$. Then $f$ is a topologically linear isomorphic between $V(D)$ and $(\mathbb{V}_{\kappa}, \mu_{\kappa})$.
Moreover, it follows from \cite[Theorem 1]{P2012} that $\tau_{\omega}=\nu_{\omega}=\mu_{\omega}$. However, if $\kappa$ is uncountable the situation changes \cite[Theorem 1]{P2012}. Furthermore, for each $n\in\mathbb{N}$, let $$\mathbb{V}_{\kappa, n}=\{(x_{\alpha})_{\alpha<\kappa}\in\mathbb{V}_{\kappa}: |\{\alpha\in\kappa: x_{\alpha}\neq 0\}|\leq n, x_{\alpha}\in[-n, n], \alpha<\kappa\}.$$
Obviously, each $\mathbb{V}_{\kappa, n}$ is closed subspace in $\mathbb{V}_{\kappa}$ within the topologies $\tau_{\kappa}$, $\nu_{\kappa}$ or $\mu_{\kappa}$.

For a subspace $Z$ of a space $X$, let $V(Z, X)$ be the vector subspace of $V(X)$ generated algebraically by $Z$. We recall the following fact from \cite{P2012}.

{\bf Fact:} Let $\kappa$ be an infinite cardinal. For each $i\in\kappa$, choose some $\lambda_{i}\in\mathbb{R}^{+}_{i}, \lambda_{i}>0$, and denote by $\mathcal{S}_{\kappa}$ the family of all subsets $\mathbb{V}_{\kappa}$ of the form $$\bigcup_{i<\kappa}([-\lambda_{i}, \lambda_{i}]\times \prod_{j<\kappa, j\neq i}\{0\}).$$For every sequence $\{S_{k}\}_{k\in\omega}$ in $\mathcal{S}_{\kappa}$, we put $$\sum_{k\in\omega}S_{k}=\bigcup_{k\in\omega}(S_{0}+\ldots+S_{k})$$ and denote by $\mathscr{N}_{\kappa}$ the family of all subsets of $\mathbb{V}_{\kappa}$ of the form $\sum_{k\in\omega}S_{k}$. Then $\mathscr{N}_{\kappa}$ is a base at 0 for $(\mathbb{V}_{\kappa}, \mu_{\kappa})$.

\begin{theorem}\label{csf}
Let $X$ be a collectionwise normal $\aleph$-space. Then the following statements are equivalent:
 \begin{enumerate}
\item $\mbox{sp}_{1}(X)$ is $csf$-countable;
\item $V(X)$ is an $\aleph_{0}$-space;
\item $X$ is separable.
\end{enumerate}
\end{theorem}

\begin{proof}
By Lemma~\ref{t00}, we have (2) $\Rightarrow$ (3). Moreover, (2) $\Rightarrow$ (1) is obvious. It suffice to prove (1) $\Rightarrow$ (3) and (3) $\Rightarrow$ (2).

(3) $\Rightarrow$ (2). Suppose that $X$ is separable, it easily check that $X$ has a countable $k$-network. Hence $V(X)$ is an $\aleph_{0}$-space by Lemma~\ref{t00}.

(1) $\Rightarrow$ (3). Assume that $\mbox{sp}_{1}(X)$ is $csf$-countable and $X$ is not separable. Since $X$ is a collectionwise normal $\aleph$-space, $X$ contains a closed discrete subspace $D$ with the cardinality of $\omega_{1}$. Then the subgroup $V(D, X)$ is group isomorphic to $\mathbb{V}_{\omega_{1}}$. Then there exists a topology $\sigma$ on $\mathbb{V}_{\omega_{1}}$ such that $V(D, X)$ is topologically isomorphic to $(\mathbb{V}_{\omega_{1}}, \sigma)$. Since $L(D)$ is topologically isomorphic to $L(D, X)$ \cite{Tkachenko} and the topology of $L(D, X)$ is coarse than $V(D, X)$, we have $\sigma$ is finer than the box topology $\tau_{\omega_{1}}$. For each $\alpha<\omega_{1}$ and $n\in \mathbb{N}$, let $a_{\alpha, n}=(x_{\beta})_{\beta<\omega_{1}}$, where $x_{\beta}=\frac{1}{n}$ if $\beta=\alpha$, and $x_{\beta}=0$ if $\beta\neq\alpha$. Put $$Y=\{0\}\cup\{a_{\alpha, n}: \alpha<\omega_{1}, n\in \mathbb{N}\}.$$ Then since $\sigma$ is finer than $\tau_{\omega_{1}}$, it is easy to see that $Y$ is a copy of $S_{\omega_{1}}$ in $(\mathbb{V}_{\omega_{1}, 1}, \sigma|_{\mathbb{V}_{\omega_{1}, 1}})$. Hence $\mbox{sp}_{1}(X)$ contains a copy of $S_{\omega_{1}}$. However, $S_{\omega_{1}}$ is not $csf$-countable, which is a contradiction.
\end{proof}

Theorem~\ref{t66} below gives a characterization of a space $X$ such that $V(X)$ is $snf$-countable, and obtain that $V(X)$ is $snf$-countable if and only if $X$ is finite.

\begin{proposition}\label{p4}
The space $V(\mathbb{N})$ is not $snf$-countable.
\end{proposition}

\begin{proof}
Since $V(\mathbb{N})$ is topologically isomorphic to $(\mathbb{V}_{\omega}, \mu_{\omega})$, it suffices to prove that $(\mathbb{V}_{\omega}, \mu_{\omega})$ is not $snf$-countable. For arbitrary $n, m\in\mathbb{N}$, let $x_{m, n}=(x_{l})_{l\in\mathbb{N}}$, where $x_{l}=\frac{1}{n}$ if $l=m$, and $x_{l}=0$ if $l\neq m$. Put $$Y=\{0\}\cup\{x_{m, n}: m, n\in\mathbb{N}\}.$$ We claim that $Y$ is a copy of $S_{\omega}$ in $V(\mathbb{N})$. It suffices to prove that for an arbitrary infinite subset $B\subset\{x_{m, n}: m, n\in\mathbb{N}\}$ with $|B\cap \{x_{m, n}: n\in\mathbb{N}\}|<\omega$ for each $m\in\mathbb{N}$ we have $0\not\in \overline{B}$. Indeed, there exists a function $\varphi: \mathbb{N}\rightarrow \mathbb{N}$ such that $B\cap \{x_{m, i}: i\geq \varphi(m)\}=\emptyset$ for each $m\in\mathbb{N}$. For each $k\in\omega$, let $$S_{k}=\bigcup_{i<\omega}([-\frac{1}{2^{k+2}\varphi(i)}, \frac{1}{2^{k+2}\varphi(i)}]\times \prod_{j<\omega, j\neq i}\{0\}).$$Let $G=\sum_{k\in\omega}S_{k}$. The set $G$ is an open neighborhood of 0 in $V(\mathbb{N})$ such that $$G\subset\prod_{i\in\omega}[-\frac{1}{2\varphi(i)}, \frac{1}{2\varphi(i)}]=U.$$ However, $U\cap B=\emptyset$. Hence $0\not\in \overline{B}$. Since $S_{\omega}$ is not $snf$-countable, $(\mathbb{V}_{\omega}, \mu_{\omega})$ is not $snf$-countable.
\end{proof}

\begin{theorem}\label{t66}
Let $X$ be a (Tychonoff) space. Then $V(X)$ is $snf$-countable if and only if $X$ is finite.
\end{theorem}

\begin{proof}
If $X$ is finite, then it is obvious. Assume that $V(X)$ is $snf$-countable and $X$ is infinite. Then it follows from \cite[Theorem 4.1]{GM2017} that $V(X)$ contains a closed vector subspace which is topologically isomorphic to $V(\mathbb{N})$. Hence $V(\mathbb{N})$ is $snf$-countable, which is a contradiction with Proposition~\ref{p4}.
\end{proof}

We shall give a characterization of a $k^*$-metrizable space $X$ such that $V(X)$ is a $k$-space, which gives an affirmative answer to Question~\ref{t1} if $X$ is a $k^*$-metrizable space. Moreover, the characterization of a non-metrizable $k^*$-metrizable space $X$ is also given such that if the $k$-property of $\mbox{sp}_{4}(X)$ implies the $k$-property of $V(X)$. First, we shall prove some results, which will be used in our proof.

\begin{theorem}
Let $X$ be a submetrizable space. Then $V(X)$ is submetrizable.
\end{theorem}

\begin{proof}
Since $X$ is submetrizable, there exists a metric space $M$ such that $i: X\rightarrow M$ is an one-to-one continuous mapping. Then the mapping $\hat{i}: V(X)\rightarrow V(M)$ is an one-to-one continuous mapping. Let $d$ be a metric of $M$, and $\widehat{d}$ the Graev extention of $d$ in $L(M)$. Then $(L(M), \widehat{d})$ is a metric space \cite{U1991}. Hence $L(M)$ is submetrizable. Since the mapping $\hat{i}: V(X)\rightarrow L(M)$ is an one-to-one continuous mapping. Thus $V(X)$ is submetrizable.
\end{proof}

\begin{corollary}\label{c0}
Let $X$ be a submetrizable space. Then $V(X)$ is a $k$-space if and only if $V(X)$ is sequential.
\end{corollary}

\begin{lemma}\label{t1}
For any uncountable cardinal $\kappa$, the tightness of $(\mathbb{V}_{\kappa, 2}, \mu_{\kappa}|_{\mathbb{V}_{\kappa, 2}})$ is uncountable.
\end{lemma}

\begin{proof}
By \cite[Theorem 20.2]{JW1997}, we can find two families $\mathscr{A}=\{A_{\alpha}: \alpha\in\omega_{1}\}$ and $\mathscr{B}=\{B_{\alpha}: \alpha\in\omega_{1}\}$ of infinite subsets of $\omega$ such that

\smallskip
(a) $A_{\alpha}\cap B_{\beta}$ is finite for all $\alpha, \beta<\omega_{1}$;

(b) for no $A\subset \omega$, all the sets $A_{\alpha}\setminus A$ and $B_{\alpha}\cap A$, $\alpha\in\omega_{1}$ are finite.

\smallskip
For arbitrary $\alpha, \beta\in \omega_{1}$ and $n\in\omega$, let $$c_{\alpha, \beta, n}=(x_{\gamma})_{\gamma<\kappa}\in\mathbb{V}_{\kappa},$$where $x_{\gamma}=\frac{1}{n+1}$ if $\gamma=\alpha\ \mbox{or}\ \beta$, and $x_{\gamma}=0$ if $\gamma\not\in\{\alpha, \beta\}$. Put $$X=\{c_{\alpha, \beta, n}: \alpha, \beta\in \omega_{1}, n\in A_{\alpha}\cap B_{\beta}\}.$$Obviously, we have $0\not\in X$. We shall prove that $0$ belongs to the closure of $X$ in $(\mathbb{V}_{\kappa, 2}, \mu_{\kappa}|_{\mathbb{V}_{\kappa, 2}})$, but $0$ does not belong to the closure of $B$ for any countable subset $B$ of $X$ in $(\mathbb{V}_{\kappa, 2}, \mu_{\kappa}|_{\mathbb{V}_{\kappa, 2}})$.

Since $(\mathbb{V}_{\kappa, 2}, \mu_{\kappa}|_{\mathbb{V}_{\kappa, 2}\mathbb{V}_{\kappa, 2}})$ is closed in $(\mathbb{V}_{\kappa}, \mu_{\kappa})$ and $X\subset \mathbb{V}_{\kappa, 2}$, we shall prove that $0$ belongs to the closure of $X$ in $(\mathbb{V}_{\kappa}, \mu_{\kappa})$. Then it suffices to prove $X\cap\sum_{k\in\omega}S_{k}\neq\emptyset$ for an arbitrary sequence $\{S_{k}\}_{k\in\omega}$ in $\mathcal{S}_{\kappa}$. Obviously, we can choose $S_{0}^{\prime}\in\mathcal{S}_{\kappa}$ such that $S_{0}^{\prime}\subset S_{0}$ and $S_{0}^{\prime}\subset S_{1}$. Then there exists a function $\varphi: \omega_{1}\rightarrow\omega$ such that $$W=\bigcup_{i<\kappa}([-\frac{1}{\varphi(i)+1}, \frac{1}{\varphi(i)+1}]\times \prod_{j<\kappa, j\neq i}\{0\})\subset S_{0}^{\prime}.$$Then $\sum_{k\in\omega}W_{k}\subset\sum_{k\in\omega}S_{k}$, where $W_{0}=W_{1}=W$ and $W_{k}=S_{k}$ for each $k>1$. We claim that $X\cap\sum_{k\in\omega}W_{k}\neq\emptyset$, hence $\sum_{k\in\omega}S_{k}\cap X\neq\emptyset$. Indeed, for each $\alpha\in\omega_{1}$, put
$$A_{\alpha}^{\prime}=\{n\in A_{\alpha}: n\geq \varphi(\alpha)+1\}\ \mbox{and}\ B_{\alpha}^{\prime}=\{n\in B_{\alpha}: n\geq \varphi(\alpha)+1\}.$$ It follows that there exist $\alpha, \beta\in\omega_{1}$ such that $A_{\alpha}^{\prime}\cap B_{\beta}^{\prime}\neq\emptyset$. If not, then we have $$\bigcup\{A_{\alpha}^{\prime}: \alpha\in\omega_{1}\}\cap \bigcup\{B_{\alpha}^{\prime}: \alpha\in\omega_{1}\}=\emptyset.$$ Put $A=\bigcup\{A_{\alpha}^{\prime}: \alpha\in\omega_{1}\}$. We get a contradiction with (b). Therefore, choose $n\in A_{\alpha}^{\prime}\cap B_{\beta}^{\prime}$. Then it is obvious that $$c_{\alpha, \beta, n}\in X\cap (W+W)\subset \sum_{k\in\omega}W_{k}\subset \sum_{k\in\omega}S_{k}.$$

Finally we prove that for an arbitrary countable subset $B$ of $X$ the point $0$ does not belongs to the closure of $B$ in $(\mathbb{V}_{\kappa}, \mu_{\kappa})$. Take an arbitrary countable subset $B$ of $X$. Then there exists an cardinal $\eta<\omega_{1}$ such that $$B\subset \{c_{\alpha, \beta, n}: \alpha, \beta<\eta, n\in\omega\}.$$ Without loss of generality, we may assume $\eta=\omega$ (otherwise order $\eta$ as $\omega$). Define a function $\psi: \omega_{1}\rightarrow\omega$ by $$\psi(m)=2+\max\{A_{k}\cap B_{l}: k, l\leq m\}$$ for each $m<\omega$ and $\psi(\alpha)=1$ if $\omega\leq\alpha<\omega_{1}$. For each $k\in\omega$, put $$G_{k}=\bigcup_{i<\kappa}([-\frac{1}{2^{k+1}\psi(i)}, \frac{1}{2^{k+1}\psi(i)}]\times \prod_{j<\kappa, j\neq i}\{0\}).$$ Then $\sum_{k\in\omega}G_{k}\in\mathscr{N}_{\kappa}$ and $\sum_{k\in\omega}G_{k}\cap B=\emptyset.$ Indeed, assume $c_{\alpha, \beta, n}\in B\cap \sum_{k\in\omega}G_{k}$, and then since $c_{\alpha, \beta, n}\in B$, we have $n\in A_{\alpha}\cap B_{\beta}$. Moreover, we can assume that $\alpha\leq\beta$. Obviously, $\psi(\beta)>\max(A_{\alpha}\cap B_{\beta})+1$, thus $\frac{1}{n+1}>\frac{1}{\psi(\beta)}$. Then it follows from the definition of each $G_{k}$ that $c_{\alpha, \beta, n}\not\in\sum_{k\in\omega}G_{k}$.
\end{proof}

\begin{lemma}\label{t2}
Let $D$ be an uncountable discrete space. Then the tightness of $\mbox{sp}_{2}(D)$ is uncountable. In particular, $V(D)$ is uncountable.
\end{lemma}

\begin{proof}
Let $D$ be the cardinality of $\kappa$. It is easy to see that $\mbox{sp}_{2}(D)$ homeomorphic to $(\mathbb{V}_{\kappa, 2}, \mu_{\kappa}|_{\mathbb{V}_{\kappa, 2}})$. By Lemma~\ref{t1}, the tightness of $\mbox{sp}_{2}(D)$ is uncountable.
\end{proof}

The following Lemma~\ref{llll6} improves a well-known result in \cite{AOP1989}. First, we recall a concept. A space is called {\it $\aleph_1$-compact} if every uncountable subset of $X$ has a cluster point.

\begin{lemma}\label{llll6}
Let $X$ be a $k^*$-metrizable space. Then $A(X)$ is a $k$-space if and only if $X$ is the topological sum of a $k_{\omega}$-space and a discrete space.
\end{lemma}

\begin{proof}
Clearly, it suffices to prove the necessity. Assume that $A(X)$ is a $k$-space, which also implies that $X^{2}$ is a $k$-space.
Since $X$ is a $k^*$-metrizable space, $X$ has a compact-countable $k$-network. Then it follows from \cite[Theorem 3.4]{LiuT1998} that $X$ is either first-countable or locally $k_{\omega}$. If $X$ is first-countable, then $X$ is metrizable since $X$ is a $k^*$-metrizable space. Then it follows from \cite{AOP1989} that $X$ is locally compact and the set of all non-isolated points $X^{\prime}$ of $X$ is separable. Therefore, it easily see that $X$ can be represented as $X=X_{0}\bigoplus D$, where $X_{0}$ is a $k_{\omega}$-space and $D$ is a discrete space.

Assume that $X$ is locally $k_{\omega}$. Since $X$ is a $k^*$-metrizable space, then there exists a compact countable $k$-network consisting of sets with compact closures, hence $X$ has a star-countable $k$-network. Then it follows from \cite[Corollary 2.4]{Sakai1997} that $X$ is a paracompact $\sigma$-space, then $A(X)$ is also a paracompact $\sigma$-space by \cite[Theorem 7.6.7]{AT2008}. Therefore, $A(X)$ is a sequential space since $A(X)$ is a paracompact $\sigma$-space. We claim that the set $X^{\prime}$ of all non-isolated points is $\aleph_{1}$-compact.

Suppose not, then there exists an uncountable closed discrete subset $\{x_{\alpha}: \alpha<\omega_{1}\}$ of $X^{\prime}$. Since $X$ is paracompact, there exists a family of discrete open subsets $\{U_{\alpha}: \alpha<\omega_{1}\}$ such that $x_{\alpha}\in U_{\alpha}$ for each $\alpha<\omega_{1}$. Since $A(X)$ is a $\sigma$ and $k$-space, the space $X$ is sequential. Then for each $\alpha<\omega_{1}$ we can take a nontrivial convergent sequence $\{x_{\alpha, n}: n\in\mathbb{N}\}\subset U_{\alpha}$ with the limit point $x_{\alpha}$. Let $Y$ be the quotient space by identifying all the points $\{x_{\alpha}: \alpha<\omega_{1}\}$ to a point $\{\infty\}$. Then the natural mapping $A(X)\rightarrow A(Y)$ is quotient, hence $A(Y)$ is a sequential space. Then it is easy to check that $Y$ contains a closed copy of $S_{\omega_{1}}$, hence $A(Y)$ contains a closed a copy of $S_{\omega_{1}}\times S_{\omega_{1}}$, which is a contradiction with \cite[Corollary 7.6.23]{AT2008}. Therefore, the set $X^{\prime}$ of all non-isolated points is $\aleph_{1}$-compact. Then $X^{\prime}$ is a Lindel\"{o}f space. Therefore, it easily check that $X$ is the topological sum of a $k_{\omega}$-space with a discrete space since $X^{\prime}$ is a Lindel\"{o}f space and $X$ is locally $k_{\omega}$.
\end{proof}

The following Lemma~\ref{La} improves a well-known result in \cite{Y1993}.

\begin{lemma}\label{La}
Let $X$ be a $k^*$-metrizable space. Then the following statements are equivalent:
 \begin{enumerate}
\item $A_{4}(X)$ is a $k$-space;
\item each $A_{n}(X)$ is a $k$-space;
\item $X$ satisfies at least one of the following conditions:\\
(a) $X$ is the topological sum of a $k_\omega$-space and a discrete space;\\
(b) $X$ is metrizable and $X^{\prime}$ is compact.
  \end{enumerate}
\end{lemma}

\begin{proof}
Obviously, we have (2) $\Rightarrow$ (1). It suffice to prove (1) $\Rightarrow$ (3) and (3) $\Rightarrow$ (2).

(3) $\Rightarrow$ (2). If $X$ is metrizable and $X^{\prime}$ is compact, then it follows from \cite[Theorem 4.2]{Y1993} that each $A_{n}(X)$ is a $k$-space.
Assume that $X$ is the topological sum of a $k_\omega$-space $Y$ and a discrete space $D$. Then $A(X)$ is topologically isomorphic to $A(Y)\times A(D)$. Then It follows from \cite[Theorem 7.41]{AT2008} that each $A_{n}(X)$ is a $k$-space.

(1) $\Rightarrow$ (3). By \cite[Theorem 4.2]{Y1993}, it suffices to prove $X$ is metrizable or the topological sum of a $k_\omega$-space and a discrete space. Assume that $X$ is not metrizable. Then we shall prove that $X$ is the topological sum of a $k_\omega$-space and a discrete space. First, we claim that $X$ contains a closed copy of $S_\omega$ or $S_2$. Suppose not, since $X$ is a $k$-space with a point-countable $k$-network, it follow from \cite[Lemma 8]{YanLin1999} and \cite[Corollary 3.10]{Lin1997} that $X$ has a point-countable base, and thus $X$ is metrizable since a paracompact $\sigma$-space with a point-countable base is metrizable \cite{Gr1984}, which is a contradiction with the assumption.
Moreover, by \cite[Lemma 4.7]{LLL2016}, $A_{4}(X)$ contains a closed copy of $X^{2}$, hence $X^{2}$ is a $k$-space. Next, we shall prove that $X$ is the topological sum of a family of $k_{\omega}$-spaces. We divide the proof into the following two cases.

 \smallskip
{\bf Case 1.1:} The space $X$ contains a closed copy of $S_\omega$.

 \smallskip
Since $S_\omega\times X$ is a closed subspace of $X^2$, the subspace $S_\omega\times X$ is a $k$-space. By \cite[Lemma 4]{Lin1998}, the space $X$ has a compact-countable $k$-network consisting of sets with
compact closures, hence $X$ has a compact-countable compact $k$-network  $\mathscr{P}$. Then $\mathscr{P}$ is star-countable, hence it follows from \cite{H1964} that we have $$\mathscr{P}=\bigcup_{\alpha\in A}\mathscr{P}_\alpha,$$ where each $\mathscr{P}_\alpha$ is countable and $(\bigcup\mathscr{P}_\alpha)\cap(\bigcup\mathscr{P}_\beta)=\varnothing$ for any $\alpha\neq\beta\in A$. For each $\alpha\in A$, put $X_\alpha=\bigcup\mathscr{P}_\alpha$. Obviously,  each $X_\alpha$ is a closed $k$-subspace of $X$ and has a countable compact $k$-network $\mathscr{P}_\alpha$. Moreover, we claim that each $X_\alpha$ is open in $X$. Indeed, fix an arbitrary $\alpha\in A$. Since $X$ is a $k$-space, it suffices to prove that $\bigcup\{X_{\beta}: \beta\in A, \beta\neq\alpha\}\cap K$ is closed in $K$ for each compact subset $K$ in $X$. Take an arbitrary compact subset $K$ in $X$. Since $\mathscr{P}$ is a $k$-network of $X$, there exists a finite subfamily $\mathscr{P}^{\prime}\subset\mathscr{P}$ such that $K\subset \bigcup\mathscr{P}^{\prime}$. Then
\begin{eqnarray}
\bigcup\{X_{\beta}: \beta\in A, \beta\neq\alpha\}\cap K&=&\bigcup\{X_{\beta}: \beta\in A, \beta\neq\alpha\}\cap K\cap\bigcup\mathscr{P}^{\prime}\nonumber\\
&=&K\cap\{P: P\in\mathscr{P}^{\prime}, P\not\in\mathscr{P}_{\alpha}\}.\nonumber
\end{eqnarray}
Since each element of $\mathscr{P}^{\prime}$ is compact, the set $\bigcup\{X_{\beta}: \beta\in A, \beta\neq\alpha\}\cap K$ is closed in $K$.
Therefore, $X=\bigoplus_{\alpha\in A}X_\alpha$ and each $X_\alpha$ is a $k_\omega$-subspace of $X$. Thus $X$ is the topological sum of a family of $k_\omega$-subspaces.

 \smallskip
{\bf Case 1.2:} The space $X$ contains a closed copy of $S_2$.

 \smallskip
Obviously, $S_2\times X$ is a $k$-space. Since $S_\omega$ is the image of $S_2$ under the perfect mapping and the $k_R$-property is preserved by the quotient mapping, $S_\omega\times X$ is a $k$-space. By Case 1.1, $X$ is the topological sum of a family of $k_\omega$-subspaces.

 \smallskip
Therefore, $X$ is the topological sum of a family of $k_\omega$-subspaces. Finally, it suffices to prove that $X^{\prime}$ is $\aleph_{1}$-compact. Indeed, this fact is easily checked by using the well-known Theorem of Yamada \cite[Theorem 3.4]{Y1993}. Then $X$ is the topological sum of a $k_\omega$-space and a discrete space.
\end{proof}

By Lemmas~\ref{llll6} and~\ref{La}, we have the following theorem.

\begin{theorem}\label{t5}
Let $X$ be a non-metrizable $k^*$-metrizable space. Then the following statements are equivalent:
 \begin{enumerate}
\item $A(X)$ is a $k$-space;
\item $A_{4}(X)$ is a $k$-space;
\item $X$ is the topological sum of a $k_\omega$-space and a discrete space.
  \end{enumerate}
\end{theorem}

Now, we can prove some of our main results in this paper.

\begin{theorem}\label{t3}
Let $X$ be a $k^*$-metrizable space. Then the following statements are equivalent:
 \begin{enumerate}
\item $V(X)$ is a $k$-space;
\item $V(X)$ is a $k_{\omega}$-space;
\item $X$ is a $k_{\omega}$-space.
  \end{enumerate}
\end{theorem}

\begin{proof}
By \cite[Theorem 3.1]{GM2017}, we have (3) $\Rightarrow$ (2), and (2) $\Rightarrow$ (1) is obvious. It suffices to prove (1) $\Rightarrow$ (3).

(1) $\Rightarrow$ (3). By Lemma~\ref{llll6}, $X$ is the topological sum of a $k_{\omega}$-space and a discrete space. Let $X=X_{0}\bigoplus D$, where $X_{0}$ is a $k_{\omega}$-space and $D$ is a discrete space. By \cite[Corollary 2.6]{GM2017}, we have $$V(X)\cong V(X_{0})\bigoplus V(D).$$ Hence $V(D)$ is a $k$-space. Then $V(D)$ is sequential by Corollary~\ref{c0}, which implies $D$ is countable by Lemma~\ref{t2}. Therefore, $X$ is a $k_{\omega}$-space.
\end{proof}

By Corollary~\ref{c0}, Theorems~\ref{t2} and~\ref{t3}, we have the following corollary.

\begin{corollary}
For an arbitrary uncountable discrete space $D$, the space $\mbox{sp}_{2}(D)$ and $V(D)$ are all not $k$-spaces.
\end{corollary}

{\bf Remark:} In \cite[Theorem 5]{P2012}, the authors said that $(\mathbb{V}_{\kappa}, \mu_{\kappa})$ is not sequential. However, the proof is wrong. Theorem~\ref{t3} shows that $(\mathbb{V}_{\kappa}, \mu_{\kappa})$ is not sequential.

The following theorem gives a characterization of a non-metrizable $k^*$-metrizable $X$ such that if $\mbox{sp}_{4}(X)$ is a $k$-space then $V(X)$ is also a $k$-space.

\begin{theorem}\label{t7}
Let $X$ be a non-metrizable $k^*$-metrizable space. The the following statements are equivalent:
 \begin{enumerate}
\item $V(X)$ is a $k$-space;
\item $V(X)$ is a $k_{\omega}$-space;
\item $\mbox{sp}_{4}(X)$ is a $k_{\omega}$-space;
\item $\mbox{sp}_{4}(X)$ is a $k$-space;
\item $X$ is a $k_{\omega}$-space.
  \end{enumerate}
\end{theorem}

\begin{proof}
Obviously, it suffices to prove that $X$ is a $k_{\omega}$-space if $\mbox{sp}_{4}(X)$ is a $k$-space. Since $A(X)$ is closed in $V(X)$ and an embedding by \cite[Proposition 5.1]{GM2017}, $A_{4}(X)$ is closed subset of $V(X)$, hence $A_{4}(X)$ is a $k$-space. By Theorem~\ref{t5}, $X$ is the topological sum of a $k_{\omega}$-space and a discrete space. Then $D$ is countable by Lemma~\ref{t2}. Therefore, $X$ is a $k_{\omega}$-space.
\end{proof}

Finally we shall prove the last main results of this paper, and prove that for a paracompact $\sigma$-and $k$-space $X$, the tightness of $V(X)$ is countable if and only if $X$ is separable, which gives an affirmative answer to Question~\ref{q2}. First of all, we prove a lemma.

\begin{lemma}\label{llll5}
Let $X$ be a paracompact, sequential space. If the tightness of $A_{4}(X)$ is countable, then the set of all non-isolated points $X$ is separable.
\end{lemma}

\begin{proof}
Assume that the set $X^{\prime}$ of all non-isolated points $X$ is not separable. Then since $X^{\prime}$ is closed in a paracompact space $X$, there exists an uncountable closed discrete subset $\{x_{\alpha}: \alpha<\omega_{1}\}$ of $X^{\prime}$. Since $X$ is paracompact, there exists a family of discrete open subsets $\{U_{\alpha}: \alpha<\omega_{1}\}$ such that $x_{\alpha}\in U_{\alpha}$ for each $\alpha<\omega_{1}$. Since $X$ is a sequential space, for each $\alpha<\omega_{1}$ we can take a nontrivial convergent sequence $\{x_{\alpha, n}: n\in\mathbb{N}\}\subset U_{\alpha}$ with the limit point $x_{\alpha}$. For each $\alpha<\omega_{1}$, put $C_{\alpha}=\{x_{\alpha}\}\cup\{x_{\alpha, n}: n\in\mathbb{N}\}$. Put $Y=\bigoplus_{\alpha<\omega_{1}}C_{\alpha}$. It follows from \cite[Theorem 4.2]{Y1997} that $A_{4}(Y)$ is not of countable tightness. However, since $X$ is paracompact and $Y$ is closed in $X$, it follows from \cite{Tkachenko} that $A(Y)$ is topologically isomorphic to $A(Y, X)$. Then $A_{4}(Y)$ is of countable tightness, which is a contradiction. Therefore, the set of all non-isolated points of $X$ is separable.
\end{proof}

\begin{theorem}\label{t4}
Let $X$ be a paracompact $\sigma$-space. If $X$ is a $k$-space, then the following statements are equivalent:
 \begin{enumerate}
\item the tightness of $V(X)$ countable;
\item the tightness of $\mbox{sp}_{4}(X)$ is countable;
\item $X$ is separable.
  \end{enumerate}
\end{theorem}

\begin{proof}
Clearly, we have (1) $\Rightarrow$ (2). It suffice to prove (3) $\Rightarrow$ (1) and (2) $\Rightarrow$ (3).

(3) $\Rightarrow$ (1). Assume that $X$ is separable, then $X$ is a cosmic space. By Corollary~\cite[Corollary 5.20]{GM2017}, $V(X)$ is a cosmic space, hence the tightness of $V(X)$ is countable.

(2) $\Rightarrow$ (3). Assume that tightness of $\mbox{sp}_{4}(X)$ is countable. Since $X$ is a $\sigma$-space, the $k$-property in $X$ is equivalent to the sequentiality of $X$. Since $A(X)$ is closed in $V(X)$ and an embedding by \cite[Proposition 5.1]{GM2017}, the tightness of $A_{4}(X)$ is countable. It follows from Lemma~\ref{llll5} that the set of all non-isolated points $X^{\prime}$ is separable. Since $X^{\prime}$ is a paracompact $\sigma$-space, $X^{\prime}$ is cosmic space. Then $X^{\prime}$ is $\aleph_{1}$-compact. By Lemma~\ref{t2}, it is easy to see that for each neighborhood $U$ of $X^{\prime}$ in $X$ we have $X\setminus U$ is countable, which also implies that $U$ is separable. Suppose not, there exists an open subset $U$ of $X^{\prime}$ such that $U$ is not separable. Then $W$ is not $\aleph_{1}$-compact since $U$ is also a $\sigma$-space. It easily check that $W$ contains an uncountable closed discrete subset $D$ of $X\setminus X^{\prime}$, then $U\setminus D$ is an open neighborhood of $X^{\prime}$. However, $U\setminus D$ is uncountable, which is a contradiction. Thus $X$ is separable.
\end{proof}

By Theorems~\ref{csf} and~\ref{t4}, we have the following corollary.

\begin{corollary}
Let $X$ be a paracompact $\aleph$-space. If $X$ is a $k$-space, then the following statements are equivalent:
 \begin{enumerate}
\item the tightness of $V(X)$ is countable;
\item the tightness of $\mbox{sp}_{_{4}}(X)$ is countable;
\item $\mbox{sp}_{1}(X)$ is $csf$-countable;
\item $V(X)$ is an $\aleph_{0}$-space;
\item $X$ separable.
  \end{enumerate}
\end{corollary}

{\bf Remarak:} By Theorems~\ref{t3} and~\ref{t4}, the tightness of $V(\mathbb{P})$ is countable; however, $V(\mathbb{P})$ is not a $k$-space, where the irrational number $\mathbb{P}$ endowed with the usual topology.

\section{Open questions}
In this section, we pose some interesting questions about the free topological vector spaces, which are still unknown to us.

It is well-known that for a closed subset $Y$ of a metrizable space $X$, we have $F(Y)$, $A(Y)$ and $L(Y)$ are topologically isomorphic to some subgroups of $F(X)$, $A(X)$ and $L(X)$ respectively. Hence we have the following question:

\begin{question}
Let $Y$ be a closed subset of metrizable space $X$. Is $V(Y)$ topologically isomorphic to a vector subspace of $V(X)$? What if $Y$ is a closed discrete subspace?
\end{question}

By Theorems~\ref{t7},~\ref{t4} and Lemma~\ref{t2}, it is natural to pose the following two questions:

\begin{question}
Let $X$ be a metrizable space. If $\mbox{sp}_{2}(X)$ is a $k$-space, is $\mbox{sp}_{4}(X)$ a $k$-space?
\end{question}

\begin{question}\label{q0}
Let $X$ be a metrizable space. If $\mbox{sp}_{4}(X)$ is a $k$-space, is each $\mbox{sp}_{n}(X)$ a $k$-space?
\end{question}

Indeed, we have the following some partial answer to Question~\ref{q0}.

\begin{theorem}
Let $X$ be a metrizable space. If $\mbox{sp}_{4}(X)$ is a $k$-space, then the set $X^{\prime}$ is compact and $X\setminus X^{\prime}$ is countable.
\end{theorem}

\begin{proof}
Indeed, it is easy to see by Lemmas~\ref{t2},~\ref{La} and the proof of (2) $\Rightarrow$ (3) in Theorem~\ref{t4}.
\end{proof}

\begin{corollary}
Let $X$ be a locally compact metrizable space. Then the following statements are equivalent:
 \begin{enumerate}
\item $V(X)$ is a $k$-space;
\item $V(X)$ is a $k_{\omega}$-space;
\item $\mbox{sp}_{4}(X)$ is a $k$-space;
\item $\mbox{sp}_{4}(X)$ is a $k_{\omega}$-space;
\item $X$ is a $k_{\omega}$-space.
  \end{enumerate}
\end{corollary}

\begin{question}
Let $X$ be a metrizable space. If $\mbox{sp}_{2}(X)$ is of countable tightness, is $\mbox{sp}_{4}(X)$ of countable tightness?
\end{question}

\begin{question}
Let $X$ be a non-metrizable $k^*$-metrizable space. If $\mbox{sp}_{2}(X)$ is a $k$-space, is $V(X)$ a $k$-space?
\end{question}


  \end{document}